	\documentclass[letterpaper, 10 pt, conference]{ieeeconf}  
	\IEEEoverridecommandlockouts                              
	
	\usepackage{graphics}      
	\usepackage{amsmath}
	\usepackage{amsfonts}
	\usepackage{subcaption}
	\usepackage{epsfig}
	\usepackage{algorithm}
	\usepackage{algpseudocode}
	\usepackage{csquotes}
	\usepackage{verbatim}
	\usepackage{booktabs}
	\usepackage{subcaption}
	\usepackage{cite}
	\usepackage{multirow}
	\usepackage{mathtools}
	\usepackage{xcolor}
	\title{\LARGE \bf
		Safe Bayesian Optimization using Interior-Point Methods - Applied to Personalized Insulin Dose Guidance
	}

	\author{Dinesh Krishnamoorthy, 
		 Francis J. Doyle III 
		\thanks{The authors are with the Harvard John A. Paulson School of Engineering and Applied Sciences, Harvard University, Allston, MA, 02134, USA. 
			{\tt\small dkrishnamoorthy@seas.harvard.edu, frank\_doyle@seas.harvard.edu}}%
	}

\newtheorem{theorem}{Theorem}
\newtheorem{assume}{Assumption}

\newtheorem{define}{Definition}
\newtheorem{lemma}{Lemma}

\renewcommand{\algorithmicrequire}{\textbf{Input:}}
\renewcommand{\algorithmicensure}{\textbf{Output:}}
\newcommand{\algrulehor}[1][.2pt]{\par\vskip.5\baselineskip\hrule height #1\par\vskip.5\baselineskip}
\newcommand*{\algrule}[1][\algorithmicindent]{\makebox[#1][l]{\hspace*{.5em}\vrule height .75\baselineskip depth .25\baselineskip}}%

\begin{document}

		\maketitle
		\thispagestyle{empty}
		\pagestyle{empty}

\begin{abstract}
	 This paper considers the problem of Bayesian optimization for systems with safety-critical constraints, where both the objective function and the constraints are unknown, but can be observed by querying the system.  In safety-critical applications, querying the system at an infeasible point can have catastrophic consequences. Such systems require a safe learning framework, such that the performance objective can be optimized while satisfying the safety-critical constraints with high probability.  In this paper we propose a safe Bayesian optimization framework  that ensures that the  points queried are always in the interior of the  partially revealed safe region, thereby guaranteeing constraint satisfaction with high probability. The proposed interior-point Bayesian optimization framework can be used with any acquisition function, making it broadly applicable. The performance of the proposed method is demonstrated using  a personalized insulin dosing application for patients with type 1 diabetes.
\end{abstract}

\section{Introduction}
Consider the following constrained optimization problem 
\begin{align}\label{Eq:Opt}
 	\min_{x\in \mathcal{X}}   \;\;\left\{ f^{0}(x) \;	| \;  f^{i}(x) \ge 0, \;\forall i = 1,\dots,m \right\}
\end{align}
where $ f^{0}: \mathcal{X} \rightarrow \mathbb{R} $ is the objective function, $ f^{i}: \mathcal{X} \rightarrow \mathbb{R} $ are the constraints, and $ \mathcal{X} \subset \mathbb{R}^{n_{x}} $ is some compact domain. 
 
	In many applications, one has to make decisions in unknown environments, where the objective function $ f^{0}(x) $ and the constraints $ f^{i}(x) $ are unknown, but can be evaluated at any arbitrary query point  $ x $ in some domain of interest $ \mathcal{X} $. The noisy observations of the cost and constraints obtained by querying the system  can be used to learn and find the optimum of the unknown black-box system.

	Bayesian optimization is one such  powerful sequential decision-making  strategy that allows us to compute the optimum by sequentially querying the zeroth order oracle (i.e, observing only the cost and constraints from the real system) in as few steps as possible \cite{mockus1978application,jones1998efficient,shahriari2015taking}. 
	This is achieved by placing a probabilistic surrogate model, typically a Gaussian process (GP) model, that is updated via Bayesian posterior updating every time a new observation is available. The  probabilistic surrogate model conditioned on the observations, are used to choose the next query point by means of an acquisition function $  \alpha(x) : \mathcal{X} \rightarrow \mathbb{R}$. Acquisition functions are typically chosen to leverage the uncertainty in the posterior model to trade-off exploration versus exploitation \cite{shahriari2015taking}. 
	To this end, Bayesian optimization allows one to find the optimum of an unknown system by systematically exploring the action space  $ \mathcal{X} $. 
	

Decision-making in many engineering problems often require satisfaction of safety-critical constraints. Evaluating any arbitrary point in  $ \mathcal{X} $ may lead to constraint violation  with catastrophic outcomes. If the constraints are known \textit{a priori}, then this can be used to determine the feasible action space $ \mathcal{F} \subseteq \mathcal{X} $. However, the challenging case arises when the constraints are  unknown, since in this case $  \mathcal{F} $ is unknown. In constrained Bayesian optimization literature, this is typically handled by placing  probabilistic surrogate models for each constraint in addition to the cost function, which are updated based on the noisy observations of the constraints.   

Initial works in the direction of constrained Bayesian optimization such as  \cite{gelbart2014constrained,gardner2014bayesian} handled the constraints by scaling the acquisition function  \textcolor[rgb]{0,0,0}{with the probability of  feasibility }\textcolor[rgb]{0,0,0}{(PF)}
\begin{equation}\label{Eq:PF}
	x_{n} = \arg \min_{x \in \mathcal{X}}  \prod_{i=1}^m \Phi^{i}(x)\alpha^0(x)
\end{equation}
where $ \alpha^0(x) $ is the acquisition function of the unconstrained optimization problem and $ \Phi^{i}(x)  $ is the cumulative density function of the $ i^{th} $ constraint GP \cite{gelbart2014constrained}. However, in this approach, the next query point may violate the constraints since the probability of feasibility is only known accurately after observing a data point. This approach  is therefore not suitable when we have safety-critical constraints.

In the presence of safety-critical constraints, 
it is highly undesirable to explore an action where the constraints would be violated. Hence, one has to carefully choose the next query point, such that the safety-critical constraints are satisfied with high probability. The authors in \cite{sui2015safe} presented the SafeOPT algorithm, where the safety-critical constraints were posed as a minimum performance requirement constraint. This was later extended to the case with arbitrary safety constraints decoupled from the performance objective in \cite{berkenkamp2021bayesian}. These approaches are based on identifying a safe set based on the constraint observations. The safe set in these approaches are enlarged using  Lipschitz continuity properties.

This paper proposes an alternative approach for safe Bayesian optimization using interior-point methods that guarantee constraint satisfaction with high probability. The approach is similar to that of  \cite{berkenkamp2021bayesian}, however, instead of using the Lipschitz continuity properties to identify the safe set at each iteration, we enforce constraint by augmenting barrier terms to the acquisition function based on the partially revealed safe region (formally defined later in Eq. \eqref{Eq:FeasibleSet}). 
Recently, the authors in \cite{pourmohamad2021bayesian}  presented a Bayesian optimization approach using barrier functions, where the barrier functions are incorporated into the Gaussian process models itself. 
However, unlike our proposed approach,  this approach 
does not guarantee constraint  satisfaction with high probability, and also involves computing the expectation of the log operator which 
requires the use of a specific acquisition function, thus limiting its application.


Guaranteeing safety critical constraints while finding the optimum in an unknown environment arises in many applications. One such application area is personalized medicine where one would like to find the optimal drug dose that must be administered to trigger a desired response. However, this is challenging, since the effect of the drug varies significantly  from one individual to the other. Moreover, ensuring patient safety while finding the optimal drug dose without overdosing further adds to the challenge.  Often times there are several physical, genetic and environmental  factors that impact the efficacy of a drug.  Models developed based on a population level are often  not well suited to find the optimal drug dosing for a particular individual.
 We posit that Bayesian optimization is well suited for personalized dose guidance based on the observed patient response in as few iterations as possible.  As in any healthcare application, patient safety is of utmost importance  for dose guidance algorithms. Therefore, incorporating safe learning in the Bayesian optimization framework is crucial in  such applications.  In this paper we will   demonstrate the use of our proposed safe Bayesian optimization to find the  optimal insulin dose to counteract the effect of meal consumption in patients with type 1 diabetes without violating the safety-critical constraints.

The reminder of the paper is organized as follows: 
The proposed safe Bayesian optimization algorithm based on interior-point method is described in Section~\ref{sec:ProposedMethod}. Conditions under which the proposed method is shown to guarantee constraint satisfaction with high probability is analyzed in  Section~\ref{sec:theory}. The proposed method is demonstrated using a personalized insulin dose guidance application in Section~\ref{sec:BolusCalc} before concluding the paper in Section~\ref{sec:conclusion}.

\section{Proposed method}\label{sec:ProposedMethod}
\subsection{Algorithm }
Given the constrained  optimization problem \eqref{Eq:Opt}, the objective is to find the global optimum using only  noisy observations of the cost $ f^0(x_{n}) $ and  constraints $ f^{i}(x_{n}) $ obtained by sequentially evaluating actions $ x_{n} \in \mathcal{X} $ at iteration $ n $. 
We model the unknown cost and constraints with $ m+1 $ independent Gaussian processes:
\begin{subequations}
	\begin{align}\label{Eq:GP}
		f^{i}(x)\sim &\; \mathcal{GP}(\mu^{i}(x),k^i(x,x')), \; \forall i \in \mathbb{I}_{0:m}
	\end{align}
\end{subequations}
where  $ \mu^i(x) = \mathbb{E}(f^{i}(x)) $ and $ k^i(x,x') = \mathbb{E}[(f^{i}(x) - \mu^i(x))(f^{i}(x') - \mu^i(x'))  ]$  for $ i=0,\dots,m $ are the mean and covariance functions of the cost and constraints.

In this work, we assume that the constraints $ f^{i}(x) $ for all $ i = 1,\dots,m $ are safety-critical, i.e., we are not allowed to evaluate any actions that would violate the constraints.  We denote the feasible set by \[ \mathcal{F} := \{ x\in \mathcal{X} \,| \, f^{i}(x)\ge 0, \; \forall i \in \mathbb{I}_{1:m}\} \]  
Since the cost and the constraints are unknown, we  do not know the feasible set $ \mathcal{F} \subseteq \mathcal{X} $ \textit{a priori}. This implies that the first action $ x_{0} $ is not guaranteed to be feasible. We therefore make the following assumption. 
\begin{assume}
	The feasible set $ \mathcal{F} $ has a non-empty interior, and there exists a known starting point $ x_{0} \in \mathcal{F} $.
\end{assume}

Note that this is a  standard assumption in the safe learning literature \textcolor[rgb]{0,0,0}{\cite{berkenkamp2021bayesian}}. In many applications this assumption is justified by the fact that actions that are safe, but not necessarily optimal are known \textit{a priori}. For example, in personalized drug dosing applications, a safe but suboptimal initial action would correspond to  a drug dose of zero.

Although we stated at the beginning of this section that the goal is to find the global optimum, this may be challenging in the presence of safety-critical constraints since this restricts how much we can explore. Therefore, we slightly modify our goal, and restate that the objective is  to find the global optimum inside the safe set that is reachable from the initial safe point $ x_{0}  \in \mathcal{F}$.

Since the acquisition function  tells us what the next query point should be, the design and choice of the acquisition function is key to ensuring  safety of the next query point. 
At iteration $ n $, the acquisition function of the unconstrained optimization, denoted by $ \alpha^0(x) $ is induced from the posterior mean and the variance conditioned on the cost observations so far. 
In the unconstrained case, the next query point would be given by $ x_{n} = \arg \min_{x} \alpha^0(x) $\footnote{\textcolor[rgb]{0,0,0}{Keeping in line with the barrier methods from  optimization and control literature, we consider a minimization problem without loss of generality.}}. In the presence of safety-critical constraints, the acquisition function must also depend on the posterior mean and variance of the constraint GPs. By conditioning the constraint GPs on the observed constraint measurements, the feasible set is partially revealed. Based on the constraints observed until iteration $ n-1 $, we define the partially revealed feasible set as
\begin{equation}\label{Eq:FeasibleSet}
	\hat{\mathcal{F}}_{n-1} := \left\{ x \in \mathcal{X} | \mu_{n-1}^i(x) - \sqrt{\beta_n^i} \sigma_{n-1}^i(x)\ge 0,\forall i \in \mathbb{I}_{1:m}\right\}
\end{equation} 
where $ \mu_{i}(x)$ and $ \sigma_i(x) $ are the posterior mean and variance of the $ i^{th} $ constraint GP, and $ {\beta_n^i} $ is the confidence level scaling parameter. Simply put, the partially revealed safe set is defined based on the lower confidence bound \textcolor[rgb]{0,0,0}{(LCB) }of the constraint GPs. From Assumption~1, $ \hat{\mathcal{F}}_{0} \subseteq  \hat{\mathcal{F}}_{n-1}$ for all $ n>1 $.
To ensure that the next query point $x_{n} $ remains in the interior of the partially revealed safe set $ \hat{\mathcal{F}}_{n-1}  $, we propose the following acquisition function 
\begin{equation}\label{Eq:BarrierAcq}
	x_{n} = \arg \min_{x \in \mathcal{X}} \; \alpha^0(x) - {\tau} \sum_{i=1}^{m}{\mathcal{B}_{\beta_{n}}^i(x) }
\end{equation}
  where $ \mathcal{B}_{\beta_{n}}^i(x) := \ln\left[\mu_{n-1}^i(x) - \sqrt{\beta^i_{n}}\sigma_{n-1}^i(x)\right]  $  is the barrier term, and $ \tau >0$ is some small user-defined barrier parameter. \textcolor[rgb]{0,0,0}{Note that the choice of $ \tau $ is similar to any  interior point methods in the standard numerical optimization literature, where $\tau$ is reduced iteratively. One could also consider using  a constant $ \tau  $ or $ \tau^n $ with $ \tau\in (0,1) $ if suitable.}
  
The key idea of our safe Bayesian optimization is then as follows: If the confidence intervals of the constraint GPs are constructed to contain the true functions $ f^{i}(x) $ with high probability, then the log-barrier term in  \eqref{Eq:BarrierAcq} ensures that the next query point $ x_{n} $ will not violate the constraints.  The probability that the true function lies within the confidence intervals depends on the value of $ \beta_{n} $ used in \eqref{Eq:FeasibleSet}, which will be discussed in the next section. 

 It can be seen that the second term in \eqref{Eq:BarrierAcq} is appended to the unconstrained acquisition function  $ \alpha^0(x) $ which can be chosen freely. For example, using GP-LCB we have
\begin{equation}\label{Eq:GP-LCB}
	\alpha^0(x) = \mu^0_{n-1}(x) - \sqrt{\beta^0_n} \sigma^0_{n-1}(x)
\end{equation}
 Any other acquisition function such as  $ \varepsilon $-greedy, probability of improvement, expected improvement, Thompson sampling, entropy search etc. \textcolor[rgb]{0,0,0}{(see  e.g. \cite{shahriari2015taking} and the references therein)} can also be used for $ \alpha^0(x) $ in \eqref{Eq:BarrierAcq}. The proposed method is summarized in Algorithm~1.

\begin{algorithm}[t]
	\caption{Safe Bayesian Optimization using interior point method.}
	\label{alg:Training}
	\begin{algorithmic}[1]
		\Require Domain $ \mathcal{X} $, initial safe point $ x_{0} \in \mathcal{F} $, dataset $ \mathcal{D}_{0} := \{(x_{0},f^0(x_{0}),\{f^i(x_{0})\}_{i=1}^m)\} $, $ m+1 $ independent GP models, $ \tau>0 $
		\algrulehor
		\For {$ n = 1,2,\dots, $}
		\State Induce any unconstrained acquisition function  $ \alpha_{0}(x) $ using $ \mathcal{GP}(\mu_{n-1}^0(x),k^0(x,x')) $, e.g. \eqref{Eq:GP-LCB}.
		\State $ \mathcal{B}_{\beta_{n}}^i(x) \leftarrow \ln\left[\mu_{n-1}^i(x) - \sqrt{\beta^i_{n}}\sigma_{n-1}^i(x)\right]  $  $ \forall i \in \mathbb{I}_{1:m} $
		\State $ x_{n} \leftarrow \arg \min_{x\in \mathcal{X}}  \alpha^0(x) - {\tau} \sum_{i=1}^{m} \mathcal{B}_{\beta_{n}}^i(x) $
		\State Query $ x_{n} $ and observe cost and constraints
		\State  $ \mathcal{D}_{n} \leftarrow \mathcal{D}_{n-1} \cup \{(x_{n},f^0(x_{n}),\{f^i(x_{n})\}_{i=1}^m)\}  $%
				\State Update the GP models by conditioning on $ \mathcal{D}_{n}  $
		\EndFor
		\algrulehor
		\Ensure $ x_{n} $
	\end{algorithmic}
\end{algorithm}
\subsection{Theoretical results} \label{sec:theory}
In this section, we analyze the conditions under which we can guarantee constraint satisfaction with high probability using the proposed algorithm.

\begin{define}[Well-calibrated model]\label{def:wellCalibrated}
	A Gaussian process  model with posterior mean and variance $ \mu_{n-1}(x) $ and $ \sigma_{n-1}(x) $ that is used to approximate a function $ f(x) $ is said to be  a well calibrated model if  the inequality 
	\begin{equation}\label{Eq:WellCalibrated}
		\left|f(x) - \mu_{n-1}(x) \right| \le\sqrt{\beta_{n} }\sigma_{n-1}(x), \; \forall x \in \mathcal{X}, \; \forall n>0
	\end{equation}
	holds  with probability at least $ 1- \delta $ for some $ \delta \in (0,1) $.
\end{define}
That is, for a well calibrated model, the confidence interval of the GP contains the true function with high probability for all $ x $, for all iterations $ n $.
\begin{assume}
	The Gaussian processes  $ f^{i}(x) \sim \mathcal{GP}(\mu^i(x),k^i(x,x')) $ for all $ i \in \mathbb{I}_{1:m} $ are well-calibrated in the sense of Definition~\ref{def:wellCalibrated}.
\end{assume}
This assumption can be satisfied by carefully choosing $ \beta_{n} $ as quantified by \cite{chowdhury2017kernelized}, which is reformulated for our problem in the following lemma. 

\begin{figure*}[t]
	\centering
	\includegraphics[width=0.99\linewidth]{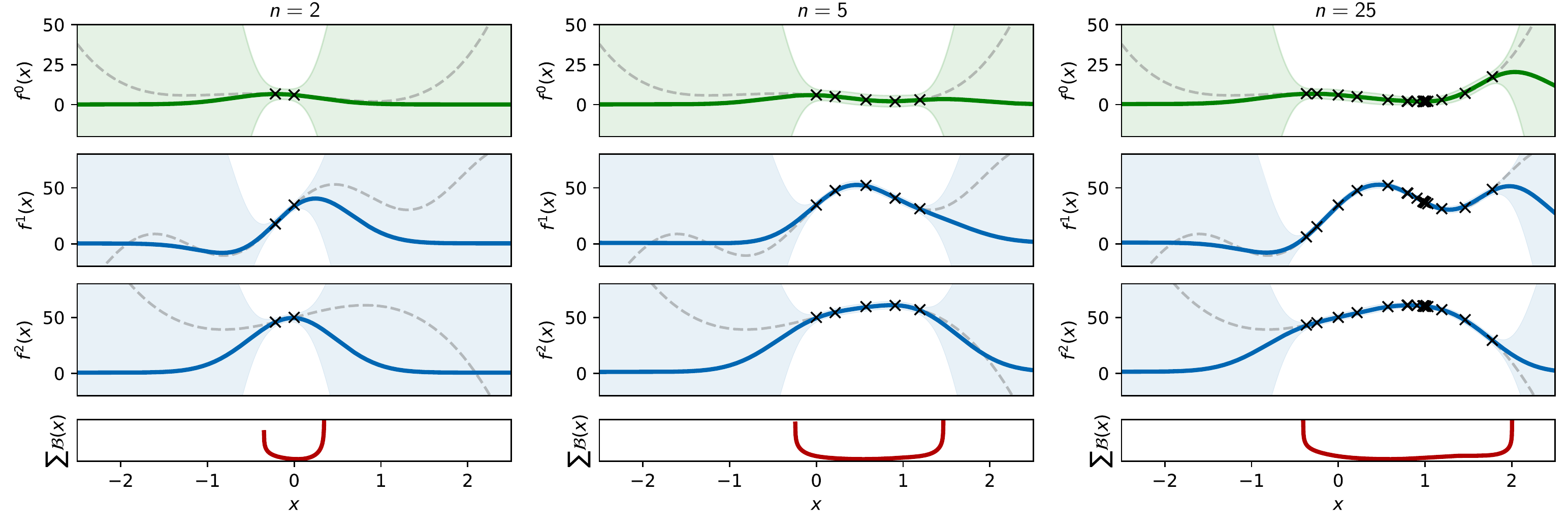}
	\caption{\textcolor[rgb]{0,0,0}{Safe Bayesian Optimization shown at iteration $ n=2,5 $, and 25. Top subplot shows the cost GP (green), the two middle subplots show the constraint GPs (blue), bottom subplot shows the barrier  term that is representative of the partially revealed safe region.} True functions are shown in gray dashed lines.  }\label{Fig:Example}
\end{figure*}
\begin{lemma}\label{lem:confidence}
	Assume that the constraints $ f^{i}(x) $ have a RKHS norm bounded by $ B^{i} $ for all $ i = 1,\dots,m $, and the corresponding  measurements are corrupted by $ v $-sub Gaussian noise.  If 
	$ \sqrt{\beta_n^i} = B^i + v\sqrt{2(\gamma^i_{n-1} + 1 + \ln(1/\delta))} $, then the following holds with probability at least $ 1- \delta $
	\begin{align}
	\left|f^{i}(x) - \mu^i_{n-1}(x) \right| \le  &\sqrt{\beta^i_{n}}\sigma^i_{n-1}(x)  \nonumber\\
	& \forall x \in \mathcal{X}, \, \forall n>0, \, \forall i \in \mathbb{I}_{1:m} 
	\end{align}
\end{lemma}
\begin{proof}
See \cite[Theorem~2]{chowdhury2017kernelized}
\end{proof}
Here $ \gamma^i_{n-1} $ is the maximum information gained after $n-1 $ iterations as explained in \cite{srinivas2009gaussian,chowdhury2017kernelized}. For a Gaussian process this would be $ \gamma_{n-1} := \max \frac{1}{2} \ln |\mathbf{I} + v^{-1}\mathbf{K}|$, which is dependent on the kernel $ k^i(x,x') $. For the sake of brevity, the reader is referred to  \cite{srinivas2009gaussian} for further details on this. 
\begin{theorem}[Safe learning]
	With the assumptions of Lemma~1, if  $ \beta_{n}^i $ is chosen as in Lemma~\ref{lem:confidence} $ \forall i \in \mathbb{I}_{1:m} $, then starting from an initial point in the interior of the feasible set $ x_{0} \in \mathcal{F}$ such that $ \hat{\mathcal{F}}_{0} \neq \emptyset $, for any choice of acquisition function $ \alpha^0(x) $, the next query point $ x_{n} $ given by \eqref{Eq:BarrierAcq} satisfies  \begin{equation}\label{Eq:theorem}
		\text{Pr} [ f^{i}(x_{n})\ge0] \ge 1- \delta, \forall i \in \mathbb{I}_{1:m}, \; \forall n >0
	\end{equation} for any $ \delta \in (0,1) $. 
\end{theorem}
\begin{proof}
Lemma~\ref{lem:confidence} implies that $\forall i \in \mathbb{I}_{1:m} $, $ \forall x \in \mathcal{X} $
\begin{equation}\label{Eq:prf1}
	f^i(x) \ge \mu_{n-1}^i(x)  - \sqrt{\beta_{n}^i} \sigma_{n-1}^i(x)  
\end{equation}
holds w.p. at least $ 1-\delta $ for any $ \delta\in (0,1) $. 

\noindent If $ \forall n>0  $, $ \exists\;  x_{n}\in  \mathcal{X} $  given by 
\eqref{Eq:BarrierAcq}, the log barrier term $ \mathcal{B}_{\beta_{n}}^i $ in \eqref{Eq:BarrierAcq} ensures 
\begin{equation}\label{Eq:prf2}
	\mu_{n-1}^i(x_{n})  - \sqrt{\beta_{n}^i} \sigma_{n-1}^i(x_{n})  > 0, \quad \forall i\in \mathbb{I}_{1:m}
\end{equation}
holds. Combining \eqref{Eq:prf1} and \eqref{Eq:prf2}  proves our result. 
\end{proof}
\subsection{Illustrative toy example}
\textcolor[rgb]{0,0,0}{We first  illustrate our approach on a toy example with two safety-critical constraints $ f^{i}(x)\ge0 $  for $ i=1,2 $ in Fig.~\ref{Fig:Example}. The cost GP is shown in green in the top subplots, and the constraint GPs are shown in blue in the middle subplots} \textcolor[rgb]{0,0,0}{(solid line denotes the mean, and the shaded area denotes the confidence interval)}.  The true (unknown) cost and constraint functions are shown in gray dashed lines. \textcolor[rgb]{0,0,0}{All the GPs use radial basis function (RBF) kernels with lengthscale = 0.5 and variance = 80, and zero prior mean.} We start from an initial safe point $ x_{0} = 0 $ such that $ \hat{\mathcal{F}}_{0}  \neq \emptyset $. The next query point is found my minimizing \eqref{Eq:BarrierAcq}. \textcolor[rgb]{0,0,0}{The log barrier term $ \sum\mathcal{B}_{\beta_{n}}^i(x) $ with $ \tau = 10^{-3} $ based on the LCB of the constraint GP is shown in the bottom subplot (in red), which is also representative of the partially revealed safe region $ \hat{\mathcal{F}}_{n-1} $. }
We see that the next query point is chosen within this safe region that minimizes the unconstrained acquisition function $ \alpha^0(x) $ (in this case  GP-LCB). As  we observe new data points, the posterior mean and variance of the constraint GPs reveal more of the feasible set, from which the next query point can be chosen.  Eventually, our algorithm converges to the optimum without violating the safety-critical constraint during the explorations.

\section{Personalized Insulin Dose Guidance using Safe Bayesian Optimization }\label{sec:BolusCalc}
\subsection{Motivation }
 Patients diagnosed with type 1 diabetes  require lifelong insulin replacement therapy, where insulin is injected subcutaneously both during mealtimes (known as bolus insulin) and fasting periods (known as basal or background insulin). 
Control algorithms, specifically model predictive control (MPC), have led to significant developments in artificial pancreas (AP) technology, which uses a continuous glucose monitor (CGM) and a continuous subcutaneous insulin infusion (CSII) pump to stabilize the  glucose levels within desirable targets \cite{doyle2014engineeringAlgorithms,boiroux2010nonlinear,hovorka2004nonlinear,bequette2013algorithms}. To this end, the  insulin needs during fasting periods are well studied  by control engineers. However, matching bolus insulin to counteract meal-related disturbance for a given patient remains an open challenge.

Following a meal consumption, the blood glucose concentration spikes up (known as \textit{postprandial } glucose).  This requires additional insulin (known as bolus). 
The bolus insulin dose is determined by a bolus calculator algorithm that is based on the carbohydrate content of the meal, and patient-specific parameters such as the insulin-to-carbohydrate ratio (ICR),  insulin sensitivity, correction factors \cite{schmidt2014bolus}, etc.
One of the main challenges with such bolus calculators is accurately determining the patient-specific parameters. Inaccurate parameters in the  bolus calculator can lead to overestimating or underestimating the bolus dosage. If the injected insulin is underestimated, this leads to \textit{hyperglycemia}, which is characterized by high blood glucose concentrations ($ > $180mg/dl), which can lead to several long term diabetes-induced complications. 
Therefore, it is important to administer sufficient insulin after each meal to avoid prolonged hyperglycemia. At the same time, overestimating the insulin can lead to \textit{hypoglycemia}, which is characterized by low blood glucose levels ($ < $70mg/dl) that can lead to severe short term or even fatal consequences. Therefore, hypoglycemia is a safety-critical constraint that must be avoided. 

Since the postprandial glucose dynamics depend on several factors, it is very challenging to develop a model that accurately predicts a patient's postprandial glucose dynamics. Therefore, there is a clear need for a model-free personalized bolus dose guidance algorithm that can learn and suggest the optimal bolus insulin dosage directly using only the CGM measurements from the patient. 

\subsection{Problem formulation }
In this section, we propose a personalized bolus calculator algorithm based on our proposed safe Bayesian optimization algorithm to find the optimum bolus dosage that is personalized to each individual while ensuring patient safety.   
\begin{figure*}
	\centering
	\includegraphics[width=0.97\linewidth]{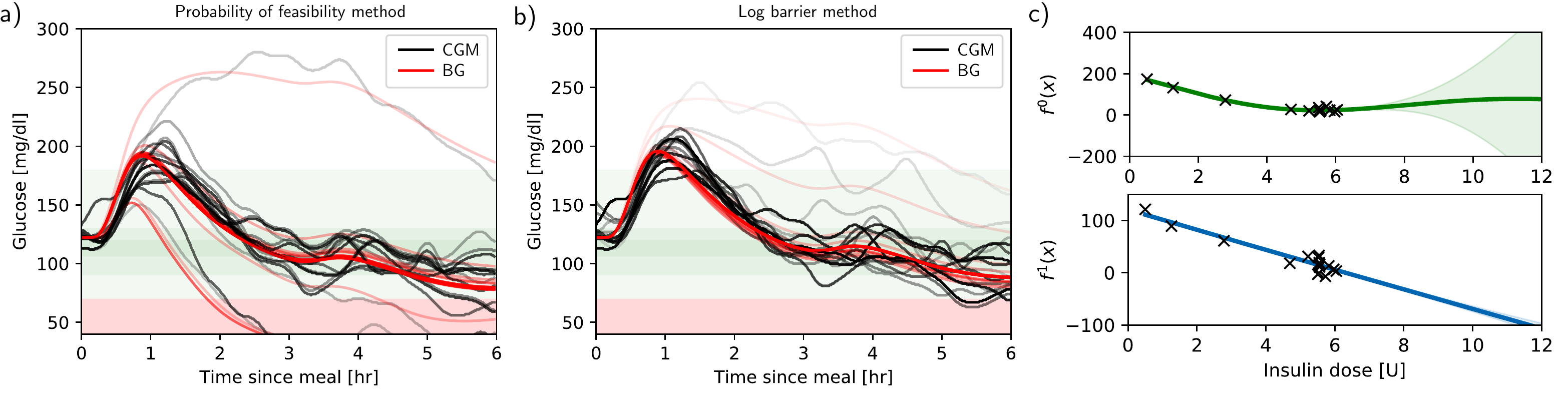}
	\caption{UVA/Padova virtual patient 1: (a) CGM (black) and plasma blood glucose (red) profiles using the probability of feasibility method \cite{gelbart2014constrained}. (b)  CGM (black) and plasma blood glucose (red) profiles using our proposed safe Bayesian optimization. (c) The cost  (green) and the constraint (blue) GPs after 15 meal observations using our proposed method.} \label{Fig:SimResults}
\end{figure*}
 To do this using safe Bayesian optimization, we must first formalize the cost and constraint functions for this application. \textcolor[rgb]{0,0,0}{In line with the existing literature on artificial pancreas \cite{van2008glycemic,cao2017extremum}, in this work we use the \textit{glycemic penalty index} (GPI)  as the metric to minimize. The GPI has been shown to be a very good scalar metric that captures the effect of glucose variations over time \cite{van2008glycemic}. The GPI is based on an} \textcolor[rgb]{0,0,0}{asymmetric penalty function $ J(y(t)) $ as  defined in \cite{cao2017extremum}}.
 If the meal is consumed at time $ t = 0 $, and a bolus insulin dose of $ x $ units is administered along with the meal, the postprandial GPI is then given as the cumulative penalty over $ T $ hours since meal consumption, i.e., in this case, the cost function is given by
\begin{equation}\label{Eq:GPI}
	f^0(x) := GPI(x) = \sum_{t=0}^{T}J(y(t))
\end{equation}
For a given meal, the postprandial glucose profile depends on the bolus insulin dose, and a smaller value of GPI indicates better glycemic control. 
If too much insulin is administered, then the blood glucose has an unavoidable undershoot \cite{goodwin2015fundamental}, leading to glucose levels dropping below the safe limit of 70mg/dl. Therefore,  the safety critical constraint is expressed in terms of the  lowest CGM value recorded after the meal peak, i.e., in this case
\begin{equation}\label{Eq:HypoConstraint}
	f^1(x) := \min \{y(t)\}_{t=t_{p}}^T - 70
\end{equation}
where $ t_{p} $ corresponds to the time of the  peak response. 
\begin{figure*}
	\centering
	\includegraphics[width=0.97\linewidth]{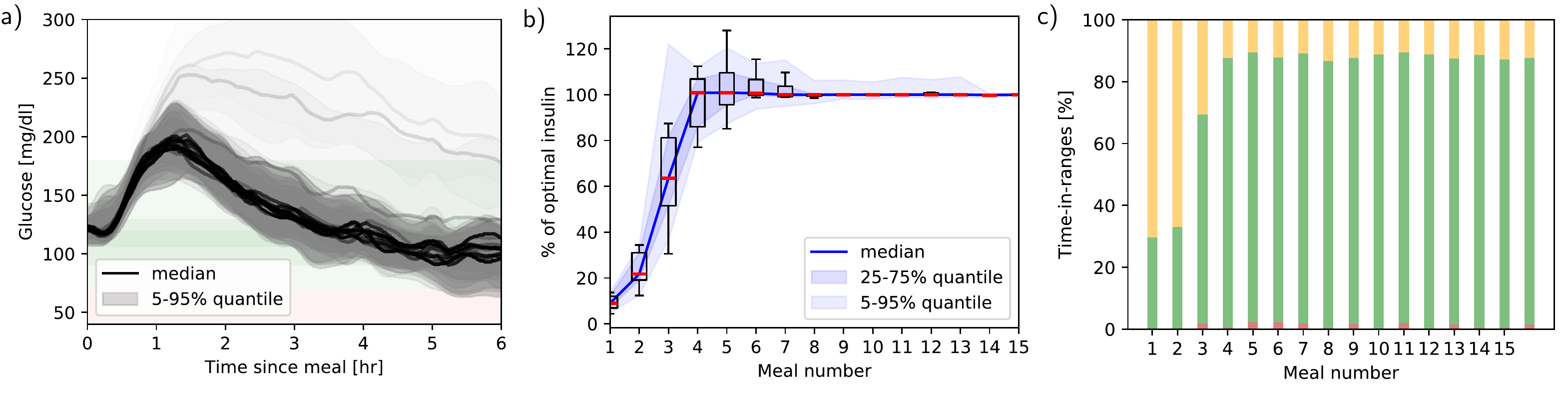}
	\caption{The overall results for the 10-adult cohort showing   (a) the postprandial CGM measurements (darker colors indicate newer meals) (b) normalized insulin value as a percentage of the optimum bolus. (c) The average \% time-in-range 70$<$CGM$\le$180mg/dl (green), time-above-range CGM$>$180mg/dl (yellow), and time-below-range BG$\le70$mg/dl (red) for the 10-adult cohort.}\label{Fig:AllPatients}
\end{figure*}

\subsection{In silico experimental setup}\label{sec:insilico}
In this work, we  use the US-FDA accepted UVA/Padova T1D metabolic simulator \cite{man2014uva} consisting of a cohort of 10-adults as our virtual patients. Note that no knowledge of the simulator model or any of its parameters are used by our algorithm. 
The \textit{in silico} experimental protocol is as follows: During each meal time,  the virtual patient is given a standardized meal portion containing  80g of carbohydrates. Along with each meal, the insulin dose suggested by the Bayesian optimization-based bolus calculator is administered to the patient. The Bayesian optimization iterations starts with the safe initial bolus insulin dose of 0.5U for all patients.  After $ T = 6 $h since the meal consumption, we observe the cost \eqref{Eq:GPI} and the constraint \eqref{Eq:HypoConstraint} based on the CGM data (which includes  sensor noise). For the period of $T= 6$h since the meal consumption, the basal insulin is kept constant at the pre-programmed value (given by the UVA/Padova simulator). We assume a budget of maximum 15 queries.  The bolus dose is chosen from a compact domain of $ x \in [0,20] $U of insulin.

For the constrained Bayesian optimization, we use two independent GP models, one  to model the effect of the insulin dose on the cost \eqref{Eq:GPI} and the other to model the  safety-critical hypoglycemia constraint \eqref{Eq:HypoConstraint}. \textcolor[rgb]{0,0,0}{We assume no prior knowledge, and therefore start with a GP with zero mean as the prior mean function for both the GPs. The cost GP uses radial basis function \textcolor[rgb]{0,0,0}{(RBF)} as the kernel, whereas the constraint GP uses both RBF and a linear kernel. }\textcolor[rgb]{0,0,0}{In this example, the hyperparameters are fixed, and are not re-optimized after each new observation. }Using the GPs, we use Algorithm~1 with $ \tau = 0.1 $ to find the optimal bolus dosage that is personalized to each patient. We also compare the performance of our algorithm with the constrained Bayesian optimization, where the acquisition function is scaled by the probability of constraint feasibility (cf. \eqref{Eq:PF}) \cite{gelbart2014constrained}. 

The safe Bayesian optimization algorithm is written in \texttt{Python}. The Gaussian process regression is performed using the \texttt{GPy} package from Sheffield ML group \cite{gpy2014}. The UVA/Padova simulator is implemented in  \texttt{MATLAB}. Note that the simulator parameters, noise levels, and their distributions are proprietary information and are accepted by the US-FDA as a substitute for  pre-clinical trials.

\subsection{Results}
Fig.~\ref{Fig:SimResults} shows the CGM profiles in black (observed measurement used by the algorithm) and the plasma blood glucose (BG) profiles in red (actual blood glucose without noise shown just for visualization) for the 15 meals  for patient 1 from the 10-adult cohort.  Older meals are shown in lighter shade. The left subplot shows the CGM profiles obtained when using the probability of feasibility method from \cite{gelbart2014constrained}. Here,  we can see that some of the insulin doses that are explored  by the PF algorithm are way too high and  violates the safety-critical hypoglycemia constraint by a large amount. Administering such large insulin doses on a real patient would be fatal. 
The middle subplot shows the results obtained when using our proposed safe Bayesian optimization approach. Here,  we see that the optimum bolus dose of around 6U of insulin is reached without violating the safety-critical constraint.  
The right subplot shows the cost and constraint GPs along with the observed data points for the 15 meals (shown in black cross) when using our proposed algorithm. 

The method is then applied to the remaining 9-adult virtual patients.
The median, 5\%-95\% percentile of the CGM data for  each meal for the 10-adult cohort  is shown in Fig.~\ref{Fig:AllPatients}a (darker colors indicate newer meals). The insulin dose normalized as a percentage of the optimal dose is  shown for the 10-adult cohort in Fig.~\ref{Fig:AllPatients}b. The percentage time-in-range averaged over the 10-adult cohort for each meal is  also shown in Fig.~\ref{Fig:AllPatients}c. Note that the time-below-range is only due to the CGM noise, and the actual blood glucose levels were above 70mg/dl for all the meals (cf. also Fig.~\ref{Fig:SimResults}b).  Detailed results for the individual patients from the 10-adult cohort (with and without CGM noise) can be made available upon request. Results from the 10-adult cohort shows that    our algorithm is able to find the optimal bolus within five meals for all the patients without causing severe hypoglycemia .
We also successfully tested  our algorithm on a larger cohort of  virtual patients  based on the Hovorka T1D simulator \cite[Ch. 2]{boiroux2012model}. However, the results are not shown here for the sake of brevity, but can be made available upon request. \textcolor[rgb]{0,0,0}{Similarly, we also compare the performance of our algorithm with that of  \cite{berkenkamp2021bayesian} and \cite{pourmohamad2021bayesian}, and can be made available upon request.  }

\section{Conclusion \& Future Work}\label{sec:conclusion}
This paper presented a safe Bayesian optimization algorithm based on interior point methods, where we showed that feasibility of the safety-critical constraints can be guaranteed with high probability by adding log-barrier terms based on the LCB of the constraint GPs to any unconstrained acquisition function. 
The proposed method was demonstrated on a personalized insulin dose guidance application, where we showed we are able to find the optimum bolus dosage within a few iterations without violating the safety-critical hypoglycemia constraint using the FDA-accepted 10-adult cohort virtual patient simulator. 
\textcolor[rgb]{0,0,0}{Future work would include extending our algorithm to safe contextual Bayesian optimization, where the optimal decisions also depend on  known disturbances,  and also study the robustness of our algorithm to unmeasured disturbances. For the personalized bolus insulin dosage application, in addition to the simulations results presented here, we have since then developed and tested this algorithm on more detailed clinically relevant simulation experiments from different metabolic simulators, also extending to contextual Bayesian optimization.  We have also tested its robustness to variations in insulin sensitivity, which will be presented in a future work. }

	\bibliographystyle{IEEEtran}  
	\bibliography{L4DC,diabetes}

\appendices

\section{Comparison to other methods from literature}

In this section, we compare the performance of our algorithm to the safe Bayesian optimization from \cite{berkenkamp2021bayesian}, and also the Bayesian optimization via barrier functions from \cite{pourmohamad2021bayesian}. 

\subsection{Illustrative example from Section~II-C}

We first consider the same empirical example from Section~II-C, with the same identical hyperparameters, starting at the same initial point using the method from \cite{berkenkamp2021bayesian}. The results are shown in Fig.~\ref{Fig:BO7}. As expected, the method from [7] also ensures constraint satisfaction. However, it was seen to be less robust to the choice of the hyperparameters with respect to convergence to the optimum. For the cost function (shown in green), there is a small band of variance in the GP model. Since [7] selects “the most uncertain element across all performance and safety functions”, the actions chosen by this algorithm chooses points from within this small band of variance. However, this method did converge to the optimum in a safe manner, when the hyperparameters are also optimized after each new observation.

Then we compare with [8], where the acquisition function is given by \[ \alpha(x) = \mu^0(x) -\sigma^0(x)^2 \sum_{i=1}^{n} \left(\ln\left[\mu^i(x) \right] - \frac{\sigma^i(x)^2}{2 \mu^i(x)^2}\right) \]
It can be seen that the method from [8] violates the constraints at iterations 4 and 6. This can be seen more clearly in Fig.~\ref{Fig:BO8} that shows the actions taken over the different iterations. Although [8] uses log barrier terms, such constraint violations are not characteristic of the barrier methods, which are supposed to strictly lie in the interior of the feasible set. Since the natural log is taken only on the mean, this can lead to potential constraint violations, unlike in our approach, where the natural log is taken on the LCB, thus enabling us to guarantee constraint satisfaction as shown in Theorem 1.

The actions taken by the different Bayesian optimization routine over the different steps are shown below in Fig.~\ref{Fig:BOcompare}, where the green shaded region represents the feasible set. Note that since the constraints are nonconvex, the feasible set is disjoint, and as stated in Section~II, the goal is to find the global optimum that is reachable within the safe set starting from the initial point.

\begin{figure}
	\centering
	\includegraphics[width=\linewidth]{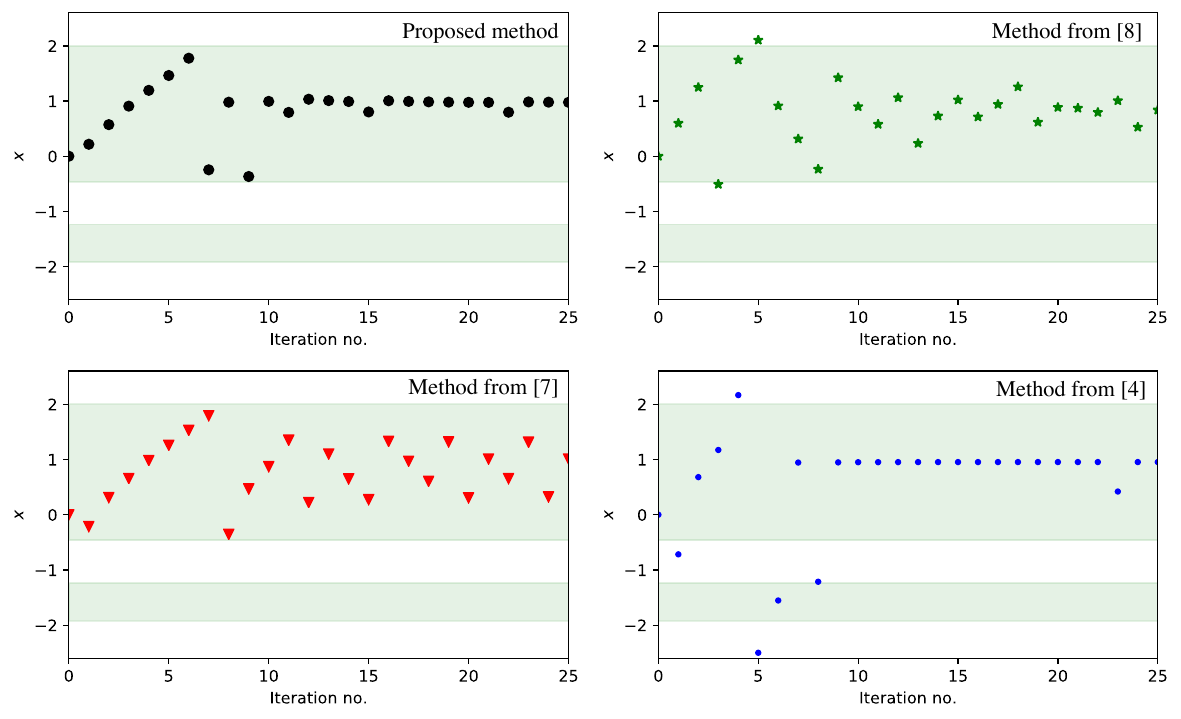}
	\caption{Comparison of the actions taken over the different iterations by the different constrained BO algorithms.}\label{Fig:BOcompare}
\end{figure}

\subsection{Personalized dose guidance example from Section~III}
We also implement the method from [7] and [8] on the insulin dose guidance application example. All the hyperparameters, initial conditions etc. are the same as used with our method. The simulation results using the algorithm from [7] is shown in Fig.~\ref{Fig:BC7}, and the method from [8] is shown in Fig.~\ref{Fig:BC8}. It can be immediately seen that, although [8] converges to the constrained optimum, it violates the safety critical constraint at iteration number 2, leading to severe hypoglycemic glucose concentrations, rendering it unsuitable for such safety critical problems.  
The method from [7] on the other hand learns the optimum dose without violating the safety critical constraint as expected. Here the hyperparameters chosen for the GP models did not lead to a case as observed in Fig.~\ref{Fig:BO7} with the empirical example, thus converging nicely to the optimum. The performance is very similar to our proposed method, which is not surprising, since both the algorithms ensure safe learning with high probability leveraging upon the results of \cite{chowdhury2017kernelized}. Since [7] chooses the next action as the most uncertain element, this indirectly has a similar effect as the LCB, thus leading to very similar performance. 

\begin{figure*}
	\begin{subfigure}{\linewidth}
		\centering
		\includegraphics[width=\linewidth]{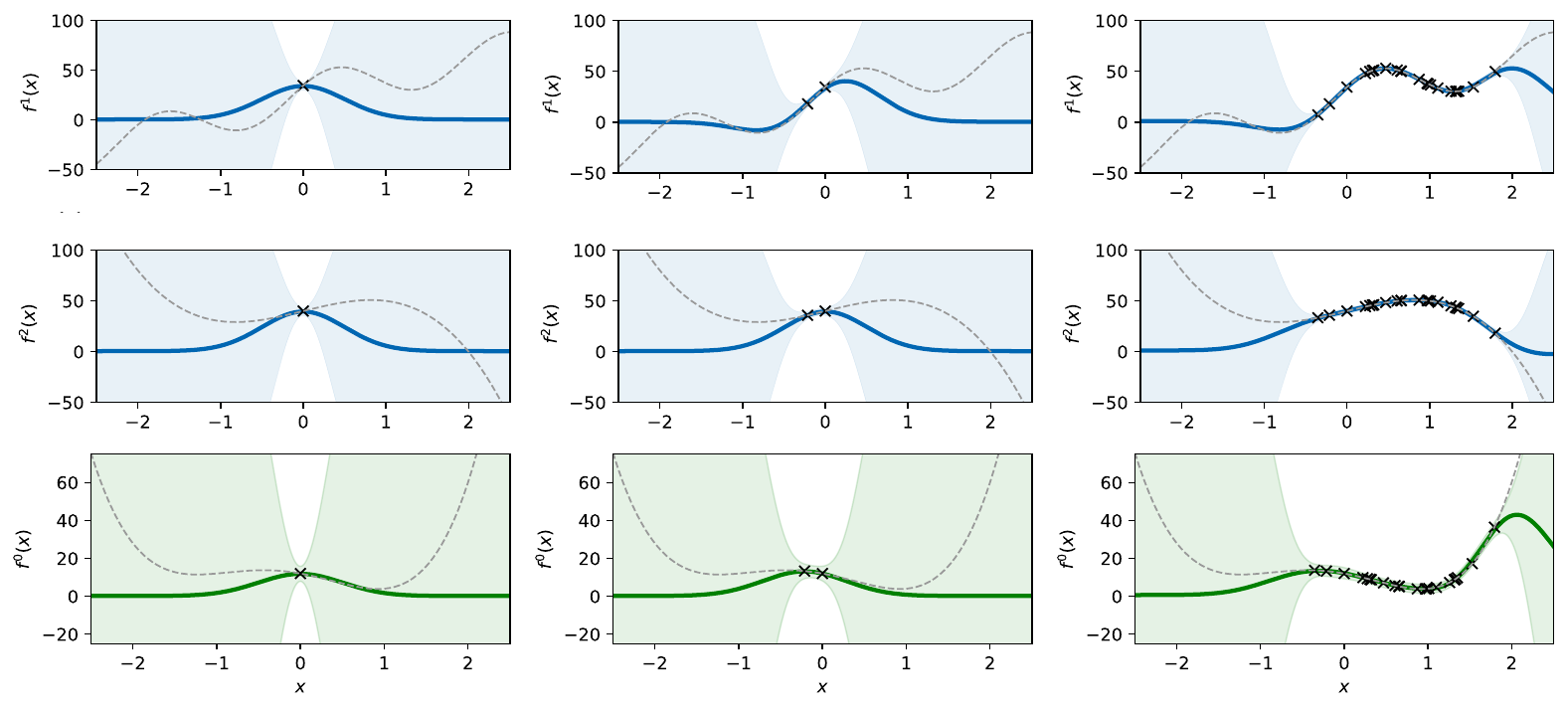}
		\caption{}\label{Fig:BO7}
	\end{subfigure}
	\begin{subfigure}{\linewidth}
		\centering
		\includegraphics[width=\linewidth]{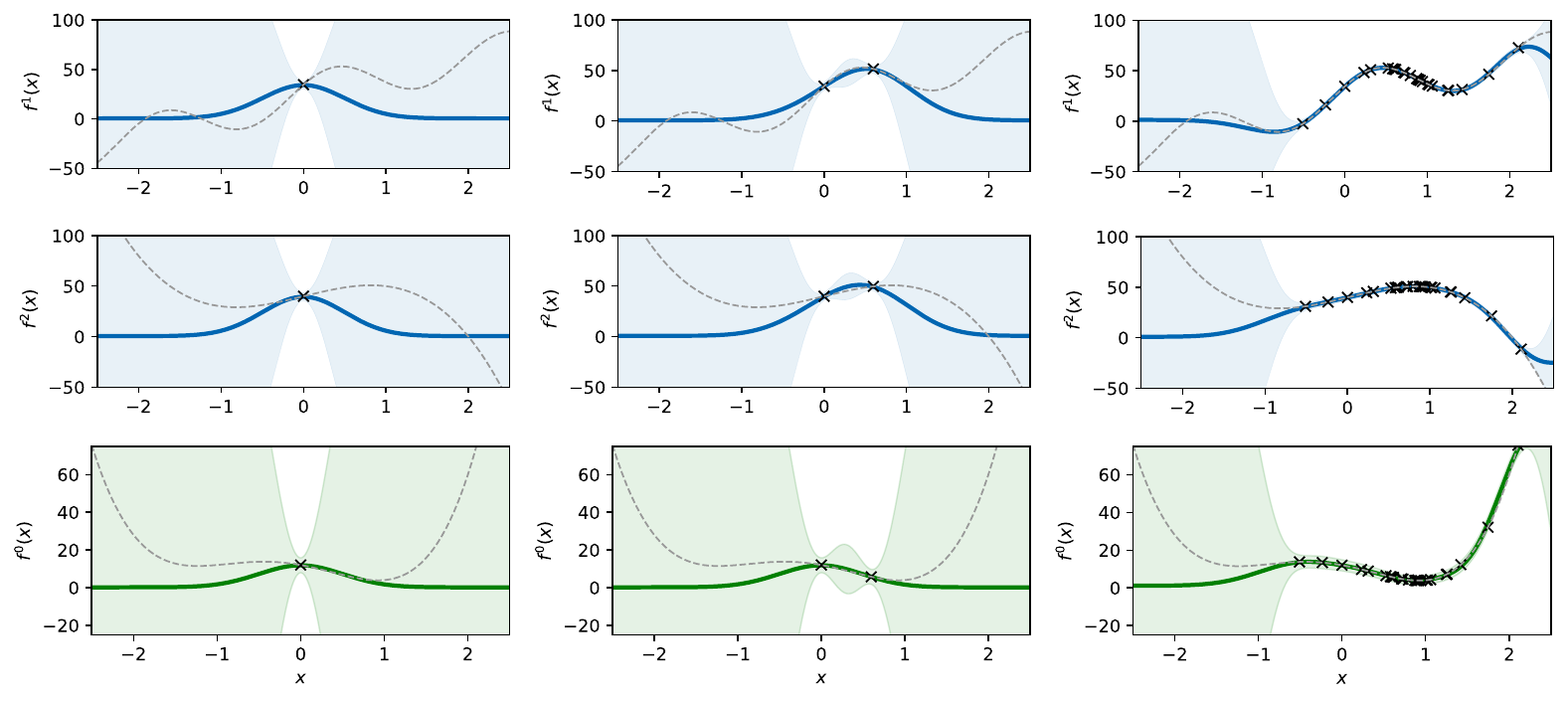}
		\caption{. }\label{Fig:BO8}
	\end{subfigure}
\caption{Cost and constraint GPs with the observed data points at iterations 1, 2, and 25 using  (a) method from [7] (b) method from [8]. The constraints are shown in blue, and the cost is shown in green. The shaded region indicates the confidence region.}
\end{figure*}

\begin{figure*}
\begin{subfigure}{\linewidth}
	\centering
	\includegraphics[width=\linewidth]{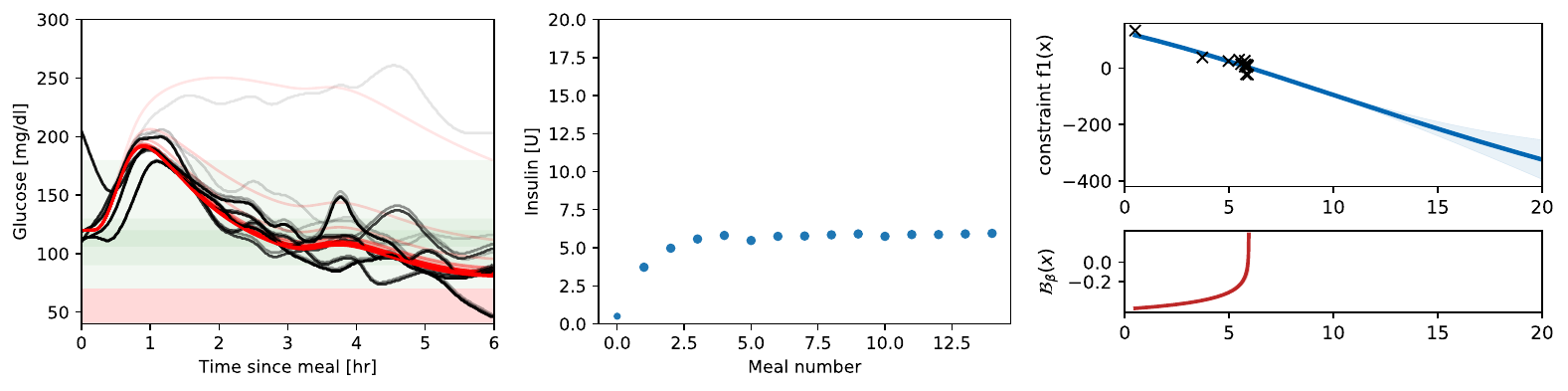}
	\caption{}\label{Fig:BC7}
\end{subfigure}
\begin{subfigure}{\linewidth}
	\centering
	\includegraphics[width=\linewidth]{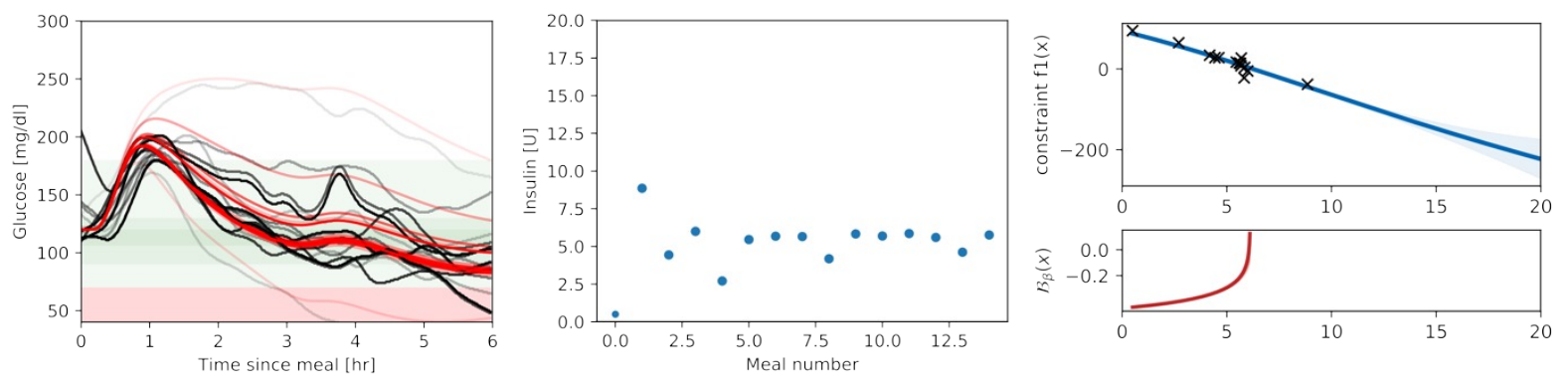}
	\caption{}\label{Fig:BC8}
\end{subfigure}
\caption{The post-meal blood glucose dynamics (left subplot), insulin dose for each meal (middle subplot), and the constraint GP after 15 meals using  (a) method from [7] (b) method from [8]. }
\end{figure*}

\subsection{Summary}
To summarize,  the probability of feasibility method from \cite{gelbart2014constrained}  converges to the constrained optimum, but violates the safety critical constraints during some explorations (cf. Fig.~2a), making it unsuitable for safety critical systems. 

Although \cite{pourmohamad2021bayesian} uses log barrier functions to find the constrained optimum, this also does not  satisfy the safety critical constraints in both the examples, although the constraint violations are smaller compared to [4]. Our method  on the other hand, satisfies the safety-critical constraints with high probability (cf.  Theorem~1). Furthermore, compared to [8], our method can be used with any acquisition function from the BO literature.

Both [7] and our proposed method guarantee safe learning with high probability, and provides very similar performance. As such, our method can be seen as an  alternative implementation to [7] that is based on the log barrier terms, instead of identifying safe sets and set of potential optimizers.

\end{document}


\maketitle

\vspace{-1.7cm}	

\subsubsection*{In silico experimental results on Hovorka virtual patients}
	We develop an algorithm to find the optimal bolus insulin dose using our proposed  safe Bayesian optimization framework. We tested this on the 10-adult cohort of the FDA-accepted UVA/Padova simulator, which is shown in the manuscript (cf. Section IV). 
	In addition to the UVA/Padova simulator, we also test our method on a different simulator based on the Hovorka T1D model [22]. 
	
		\noindent Comparing the performance of our algorithm  on a virtual patient simulator based on completely different physiological models clearly  demonstrates the robustness of our  algorithm towards patient variability, and  demonstrates the main advantage of a model-free dose guidance algorithm.
	
	\noindent  We create a cohort of 50 virtual patients based on the Hovorka T1D model, and test our algorithm with the identical tuning parameters for all the virtual patients. We use the same \textit{in silico} experimental protocol as for the UVA/Padova virtual patients. The simulator model and its parameters are the same as in [22, Ch. 2]. In this simulator, we use the plasma blood glucose as the measurement (i.e. we do not consider  sensor noise from the continuous glucose monitor).

		\noindent The left subplot of Fig. 1 shows the overall results from the cohort of 50 virtual patients  showing the median and the 5-95\% percentile of  the postprandial glucose measurements for each meal (darker colors indicate newer meals). The middle subplot shows the normalized insulin value as a percentage of the optimum bolus for the 15 meals.  The average \% time-in-range 70$<$CGM)$\le$180mg/dl (green), time-above-range CGM$>$180mg/dl (yellow), and time-below-range BG$\le70$mg/dl (red) for the cohort of 50 virtual patients is shown in the right subplot.
	  
	  	\noindent The postprandial glucose measurements, as well as the insulin dose suggested by our algorithm for each meal are shown seperately for each patient in Fig. 2 and Fig. 3 respectively. These clearly demonstrate that our algorithm is able to learn the patients optimum bolus insulin dose without violating the safety-critical hypoglycemia constraint. 

	\begin{figure}[h]
		\centering
		\includegraphics[width=0.32\linewidth]{figures/data_Hovorka/CGM_Population.pdf}
		\includegraphics[width=0.32\linewidth]{figures/data_Hovorka/Insulin_Population.pdf}
		\includegraphics[width=0.32\linewidth]{figures/data_Hovorka/TIR_Population.pdf}
	\caption{Overall results from the cohort of 50 virtual patients based on the Hovorka T1D simulator.}
	\end{figure}
	
\begin{figure*}
	\centering
	\begin{subfigure}{\linewidth}
		\centering
		\includegraphics[width=0.19\linewidth]{figures/data_Hovorka/BG_profile_Hovorka_1.pdf}
		\includegraphics[width=0.19\linewidth]{figures/data_Hovorka/BG_profile_Hovorka_2.pdf}
		\includegraphics[width=0.19\linewidth]{figures/data_Hovorka/BG_profile_Hovorka_3.pdf}
		\includegraphics[width=0.19\linewidth]{figures/data_Hovorka/BG_profile_Hovorka_4.pdf}
		\includegraphics[width=0.19\linewidth]{figures/data_Hovorka/BG_profile_Hovorka_5.pdf}
\end{subfigure}
	\begin{subfigure}{\linewidth}
	\centering
	\includegraphics[width=0.19\linewidth]{figures/data_Hovorka/BG_profile_Hovorka_6.pdf}
	\includegraphics[width=0.19\linewidth]{figures/data_Hovorka/BG_profile_Hovorka_7.pdf}
	\includegraphics[width=0.19\linewidth]{figures/data_Hovorka/BG_profile_Hovorka_8.pdf}
	\includegraphics[width=0.19\linewidth]{figures/data_Hovorka/BG_profile_Hovorka_9.pdf}
	\includegraphics[width=0.19\linewidth]{figures/data_Hovorka/BG_profile_Hovorka_10.pdf}
\end{subfigure}
	\begin{subfigure}{\linewidth}
	\centering
	\includegraphics[width=0.19\linewidth]{figures/data_Hovorka/BG_profile_Hovorka_11.pdf}
	\includegraphics[width=0.19\linewidth]{figures/data_Hovorka/BG_profile_Hovorka_12.pdf}
	\includegraphics[width=0.19\linewidth]{figures/data_Hovorka/BG_profile_Hovorka_13.pdf}
	\includegraphics[width=0.19\linewidth]{figures/data_Hovorka/BG_profile_Hovorka_14.pdf}
	\includegraphics[width=0.19\linewidth]{figures/data_Hovorka/BG_profile_Hovorka_15.pdf}
\end{subfigure}
\begin{subfigure}{\linewidth}
	\centering
	\includegraphics[width=0.19\linewidth]{figures/data_Hovorka/BG_profile_Hovorka_16.pdf}
	\includegraphics[width=0.19\linewidth]{figures/data_Hovorka/BG_profile_Hovorka_17.pdf}
	\includegraphics[width=0.19\linewidth]{figures/data_Hovorka/BG_profile_Hovorka_18.pdf}
	\includegraphics[width=0.19\linewidth]{figures/data_Hovorka/BG_profile_Hovorka_19.pdf}
	\includegraphics[width=0.19\linewidth]{figures/data_Hovorka/BG_profile_Hovorka_20.pdf}
\end{subfigure}
	\begin{subfigure}{\linewidth}
	\centering
	\includegraphics[width=0.19\linewidth]{figures/data_Hovorka/BG_profile_Hovorka_21.pdf}
	\includegraphics[width=0.19\linewidth]{figures/data_Hovorka/BG_profile_Hovorka_22.pdf}
	\includegraphics[width=0.19\linewidth]{figures/data_Hovorka/BG_profile_Hovorka_23.pdf}
	\includegraphics[width=0.19\linewidth]{figures/data_Hovorka/BG_profile_Hovorka_24.pdf}
	\includegraphics[width=0.19\linewidth]{figures/data_Hovorka/BG_profile_Hovorka_25.pdf}
\end{subfigure}
\begin{subfigure}{\linewidth}
	\centering
	\includegraphics[width=0.19\linewidth]{figures/data_Hovorka/BG_profile_Hovorka_26.pdf}
	\includegraphics[width=0.19\linewidth]{figures/data_Hovorka/BG_profile_Hovorka_27.pdf}
	\includegraphics[width=0.19\linewidth]{figures/data_Hovorka/BG_profile_Hovorka_28.pdf}
	\includegraphics[width=0.19\linewidth]{figures/data_Hovorka/BG_profile_Hovorka_29.pdf}
	\includegraphics[width=0.19\linewidth]{figures/data_Hovorka/BG_profile_Hovorka_30.pdf}
\end{subfigure}
	\begin{subfigure}{\linewidth}
	\centering
	\includegraphics[width=0.19\linewidth]{figures/data_Hovorka/BG_profile_Hovorka_31.pdf}
	\includegraphics[width=0.19\linewidth]{figures/data_Hovorka/BG_profile_Hovorka_32.pdf}
	\includegraphics[width=0.19\linewidth]{figures/data_Hovorka/BG_profile_Hovorka_33.pdf}
	\includegraphics[width=0.19\linewidth]{figures/data_Hovorka/BG_profile_Hovorka_34.pdf}
	\includegraphics[width=0.19\linewidth]{figures/data_Hovorka/BG_profile_Hovorka_35.pdf}
\end{subfigure}
\begin{subfigure}{\linewidth}
	\centering
	\includegraphics[width=0.19\linewidth]{figures/data_Hovorka/BG_profile_Hovorka_36.pdf}
	\includegraphics[width=0.19\linewidth]{figures/data_Hovorka/BG_profile_Hovorka_37.pdf}
	\includegraphics[width=0.19\linewidth]{figures/data_Hovorka/BG_profile_Hovorka_38.pdf}
	\includegraphics[width=0.19\linewidth]{figures/data_Hovorka/BG_profile_Hovorka_39.pdf}
	\includegraphics[width=0.19\linewidth]{figures/data_Hovorka/BG_profile_Hovorka_40.pdf}
\end{subfigure}
	\begin{subfigure}{\linewidth}
	\centering
	\includegraphics[width=0.19\linewidth]{figures/data_Hovorka/BG_profile_Hovorka_41.pdf}
	\includegraphics[width=0.19\linewidth]{figures/data_Hovorka/BG_profile_Hovorka_42.pdf}
	\includegraphics[width=0.19\linewidth]{figures/data_Hovorka/BG_profile_Hovorka_43.pdf}
	\includegraphics[width=0.19\linewidth]{figures/data_Hovorka/BG_profile_Hovorka_44.pdf}
	\includegraphics[width=0.19\linewidth]{figures/data_Hovorka/BG_profile_Hovorka_45.pdf}
\end{subfigure}
\begin{subfigure}{\linewidth}
	\centering
	\includegraphics[width=0.19\linewidth]{figures/data_Hovorka/BG_profile_Hovorka_46.pdf}
	\includegraphics[width=0.19\linewidth]{figures/data_Hovorka/BG_profile_Hovorka_47.pdf}
	\includegraphics[width=0.19\linewidth]{figures/data_Hovorka/BG_profile_Hovorka_48.pdf}
	\includegraphics[width=0.19\linewidth]{figures/data_Hovorka/BG_profile_Hovorka_49.pdf}
	\includegraphics[width=0.19\linewidth]{figures/data_Hovorka/BG_profile_Hovorka_50.pdf}
\end{subfigure}
\caption{Postprandial glucose dynamics for the cohort of 50 Hovorka virtual patients. Brighter shade indicates newer meals.}
\end{figure*}

\begin{figure*}
	\centering
	\begin{subfigure}{\linewidth}
		\centering
		\includegraphics[width=0.19\linewidth]{figures/data_Hovorka/Insulin_1.pdf}
		\includegraphics[width=0.19\linewidth]{figures/data_Hovorka/Insulin_2.pdf}
		\includegraphics[width=0.19\linewidth]{figures/data_Hovorka/Insulin_3.pdf}
		\includegraphics[width=0.19\linewidth]{figures/data_Hovorka/Insulin_4.pdf}
		\includegraphics[width=0.19\linewidth]{figures/data_Hovorka/Insulin_5.pdf}
	\end{subfigure}
	\begin{subfigure}{\linewidth}
		\centering
		\includegraphics[width=0.19\linewidth]{figures/data_Hovorka/Insulin_6.pdf}
		\includegraphics[width=0.19\linewidth]{figures/data_Hovorka/Insulin_7.pdf}
		\includegraphics[width=0.19\linewidth]{figures/data_Hovorka/Insulin_8.pdf}
		\includegraphics[width=0.19\linewidth]{figures/data_Hovorka/Insulin_9.pdf}
		\includegraphics[width=0.19\linewidth]{figures/data_Hovorka/Insulin_10.pdf}
	\end{subfigure}
	\begin{subfigure}{\linewidth}
		\centering
		\includegraphics[width=0.19\linewidth]{figures/data_Hovorka/Insulin_11.pdf}
		\includegraphics[width=0.19\linewidth]{figures/data_Hovorka/Insulin_12.pdf}
		\includegraphics[width=0.19\linewidth]{figures/data_Hovorka/Insulin_13.pdf}
		\includegraphics[width=0.19\linewidth]{figures/data_Hovorka/Insulin_14.pdf}
		\includegraphics[width=0.19\linewidth]{figures/data_Hovorka/Insulin_15.pdf}
	\end{subfigure}
	\begin{subfigure}{\linewidth}
		\centering
		\includegraphics[width=0.19\linewidth]{figures/data_Hovorka/Insulin_16.pdf}
		\includegraphics[width=0.19\linewidth]{figures/data_Hovorka/Insulin_17.pdf}
		\includegraphics[width=0.19\linewidth]{figures/data_Hovorka/Insulin_18.pdf}
		\includegraphics[width=0.19\linewidth]{figures/data_Hovorka/Insulin_19.pdf}
		\includegraphics[width=0.19\linewidth]{figures/data_Hovorka/Insulin_20.pdf}
	\end{subfigure}
	\begin{subfigure}{\linewidth}
		\centering
		\includegraphics[width=0.19\linewidth]{figures/data_Hovorka/Insulin_21.pdf}
		\includegraphics[width=0.19\linewidth]{figures/data_Hovorka/Insulin_22.pdf}
		\includegraphics[width=0.19\linewidth]{figures/data_Hovorka/Insulin_23.pdf}
		\includegraphics[width=0.19\linewidth]{figures/data_Hovorka/Insulin_24.pdf}
		\includegraphics[width=0.19\linewidth]{figures/data_Hovorka/Insulin_25.pdf}
	\end{subfigure}
	\begin{subfigure}{\linewidth}
		\centering
		\includegraphics[width=0.19\linewidth]{figures/data_Hovorka/Insulin_26.pdf}
		\includegraphics[width=0.19\linewidth]{figures/data_Hovorka/Insulin_27.pdf}
		\includegraphics[width=0.19\linewidth]{figures/data_Hovorka/Insulin_28.pdf}
		\includegraphics[width=0.19\linewidth]{figures/data_Hovorka/Insulin_29.pdf}
		\includegraphics[width=0.19\linewidth]{figures/data_Hovorka/Insulin_30.pdf}
	\end{subfigure}
	\begin{subfigure}{\linewidth}
		\centering
		\includegraphics[width=0.19\linewidth]{figures/data_Hovorka/Insulin_31.pdf}
		\includegraphics[width=0.19\linewidth]{figures/data_Hovorka/Insulin_32.pdf}
		\includegraphics[width=0.19\linewidth]{figures/data_Hovorka/Insulin_33.pdf}
		\includegraphics[width=0.19\linewidth]{figures/data_Hovorka/Insulin_34.pdf}
		\includegraphics[width=0.19\linewidth]{figures/data_Hovorka/Insulin_35.pdf}
	\end{subfigure}
	\begin{subfigure}{\linewidth}
		\centering
		\includegraphics[width=0.19\linewidth]{figures/data_Hovorka/Insulin_36.pdf}
		\includegraphics[width=0.19\linewidth]{figures/data_Hovorka/Insulin_37.pdf}
		\includegraphics[width=0.19\linewidth]{figures/data_Hovorka/Insulin_38.pdf}
		\includegraphics[width=0.19\linewidth]{figures/data_Hovorka/Insulin_39.pdf}
		\includegraphics[width=0.19\linewidth]{figures/data_Hovorka/Insulin_40.pdf}
	\end{subfigure}
	\begin{subfigure}{\linewidth}
		\centering
		\includegraphics[width=0.19\linewidth]{figures/data_Hovorka/Insulin_41.pdf}
		\includegraphics[width=0.19\linewidth]{figures/data_Hovorka/Insulin_42.pdf}
		\includegraphics[width=0.19\linewidth]{figures/data_Hovorka/Insulin_43.pdf}
		\includegraphics[width=0.19\linewidth]{figures/data_Hovorka/Insulin_44.pdf}
		\includegraphics[width=0.19\linewidth]{figures/data_Hovorka/Insulin_45.pdf}
	\end{subfigure}
	\begin{subfigure}{\linewidth}
		\centering
		\includegraphics[width=0.19\linewidth]{figures/data_Hovorka/Insulin_46.pdf}
		\includegraphics[width=0.19\linewidth]{figures/data_Hovorka/Insulin_47.pdf}
		\includegraphics[width=0.19\linewidth]{figures/data_Hovorka/Insulin_48.pdf}
		\includegraphics[width=0.19\linewidth]{figures/data_Hovorka/Insulin_49.pdf}
		\includegraphics[width=0.19\linewidth]{figures/data_Hovorka/Insulin_50.pdf}
	\end{subfigure}
	\caption{Bolus insulin over 5 days (i.e., 15 meals) for the cohort of 50 Hovorka virtual patients.}
\end{figure*}

\subsubsection*{Additional plots for  results from the UVA/Padova simulator}
	\noindent For the sake of brevity, we only showed the overall results from the 10-adult cohort of the FDA-accepted UVA/Padova T1D simulator in the manuscript (cf. Fig.~3 in the manuscript). Here, we show the detailed plots of the postprandial glucose dynamics (Fig.~4) and the insulin dose (Fig.~5) suggested by our algorithm for each individual virtual patient from the 10-adult cohort of the UVA/Padova simulator.

\begin{figure*}[h]
	\centering
	\begin{subfigure}{\linewidth}
		\centering
		\includegraphics[width=0.19\linewidth]{figures/data_UVAPadova/Context_CGM_BG_profiles_1.pdf}
		\includegraphics[width=0.19\linewidth]{figures/data_UVAPadova/Context_CGM_BG_profiles_2.pdf}
		\includegraphics[width=0.19\linewidth]{figures/data_UVAPadova/Context_CGM_BG_profiles_3.pdf}
		\includegraphics[width=0.19\linewidth]{figures/data_UVAPadova/Context_CGM_BG_profiles_4.pdf}
		\includegraphics[width=0.19\linewidth]{figures/data_UVAPadova/Context_CGM_BG_profiles_5.pdf}
	\end{subfigure}
	\begin{subfigure}{\linewidth}
		\centering
		\includegraphics[width=0.19\linewidth]{figures/data_UVAPadova/Context_CGM_BG_profiles_6.pdf}
		\includegraphics[width=0.19\linewidth]{figures/data_UVAPadova/Context_CGM_BG_profiles_7.pdf}
		\includegraphics[width=0.19\linewidth]{figures/data_UVAPadova/Context_CGM_BG_profiles_8.pdf}
		\includegraphics[width=0.19\linewidth]{figures/data_UVAPadova/Context_CGM_BG_profiles_9.pdf}
		\includegraphics[width=0.19\linewidth]{figures/data_UVAPadova/Context_CGM_BG_profiles_10.pdf}
	\end{subfigure}
\caption{Postprandial glucose dynamics for the cohort of 10 UVA/Padova virtual patients. Brighter shade indicates newer meals.}
\end{figure*}	
\begin{figure*}[h]
	\centering
	\begin{subfigure}{\linewidth}
		\centering
		\includegraphics[width=0.19\linewidth]{figures/data_UVAPadova/Insulin_1.pdf}
		\includegraphics[width=0.19\linewidth]{figures/data_UVAPadova/Insulin_2.pdf}
		\includegraphics[width=0.19\linewidth]{figures/data_UVAPadova/Insulin_3.pdf}
		\includegraphics[width=0.19\linewidth]{figures/data_UVAPadova/Insulin_4.pdf}
		\includegraphics[width=0.19\linewidth]{figures/data_UVAPadova/Insulin_5.pdf}
	\end{subfigure}
	\begin{subfigure}{\linewidth}
		\centering
		\includegraphics[width=0.19\linewidth]{figures/data_UVAPadova/Insulin_6.pdf}
		\includegraphics[width=0.19\linewidth]{figures/data_UVAPadova/Insulin_7.pdf}
		\includegraphics[width=0.19\linewidth]{figures/data_UVAPadova/Insulin_8.pdf}
		\includegraphics[width=0.19\linewidth]{figures/data_UVAPadova/Insulin_9.pdf}
		\includegraphics[width=0.19\linewidth]{figures/data_UVAPadova/Insulin_10.pdf}
	\end{subfigure}
\caption{Bolus insulin over 5 days (i.e., 15 meals) for the cohort of 10 UVA/Padova virtual patients.}
\end{figure*}